\documentclass{article}
\usepackage{latexsym}
\usepackage{amssymb}
\usepackage{amsmath}
\usepackage{color}
\usepackage{graphicx}
\usepackage{amsthm}
\usepackage{indentfirst}
\usepackage{mathrsfs}
\usepackage{dsfont}
\usepackage{enumitem}
\usepackage{multirow}
\usepackage{extpfeil,extarrows}
\usepackage{galois}
\usepackage{float}
\usepackage{longtable}
\allowdisplaybreaks

\newtheorem{theorem}{\color{black}\indent Theorem}[section]

\newtheorem{lemma}{\color{black}\indent Lemma}[section]
\newtheorem{remark}{\color{black}\indent Remark}[section]
\newtheorem{definition}{\color{black}\indent Definition}[section]
\newtheorem{cor}{\color{black}\indent Corollary}[section]

\newcommand{\wconv}{\xrightarrow{~w~}}
\newcommand{\convid}{\xrightarrow{~d~}}
\newcommand{\equivid}{\xlongequal{~d~}}
\DeclareMathOperator{\trace}{{\bf tr}}
\DeclareMathOperator{\hess}{{\bf Hess}}
\DeclareMathOperator{\jac}{{\bf J}}
\DeclareMathOperator{\diag}{diag}
\begin{document}

\title{Affine Periodic Solutions of Stochastic Differential Equations \footnote{This work was supported by National Basic Research Program of China (grant No. 2013CB834100), National Natural Science Foundation of China (grant No. 11901231), National Natural Science Foundation of China (grant No. 11901080),  National Natural Science Foundation of China (grant No. 11571065) and National Natural Science Foundation of China (grant No. 11171132).}}

\author{ Xiaomeng Jiang$^a$\thanks{E-mail address: jxmlucy@hotmail.com}, ~ Xue Yang$^{a,b}$\thanks{E-mail address: yangxuemath@163.com}, ~ Yong Li$^{a,b}$\thanks{E-mail address: liyongmath@163.com}~\footnote{Corresponding author}
\\
{$^a$College of Mathematics, Jilin University,}
\\
{ Changchun 130012, P. R. China.}
\\
{$^b$School of Mathematics and statistics and}
\\
{Center for Mathematics and Interdisciplinary Sciences, }
\\
{Northeast Normal University, Changchun 130024, P. R. China}
}
\date{}

\maketitle

\begin{abstract}
For stochastic affine periodic systems, we establish a law of large numbers including Halanay-type criterion and a LaSalle-type stationary oscillation principle to obtain the existence and stability of affine periodic solutions in distribution. As applications, we show the existence and asymptotic stability of stochastic affine periodic solutions in distribution via Lyapunov's method.

\end{abstract}

\section{Introduction}

In this paper, we consider following stochastic differential equation
\begin{equation}\label{sde}
dX(t)=f(t,X(t))dt+g(t,X(t))dW(t),
\end{equation}
where $\{W(t)\}$ is an $m$-dimensional standard Brownian motion. This system is called a stochastic $(Q,T)$-affine periodic system if functions $f:\mathbb{R}^+\times\mathbb{R}^l\rightarrow\mathbb{R}^l$ and $g:\mathbb{R}^+\times\mathbb{R}^l\rightarrow\mathbb{R}^{l\times m}$ are Borel-measurable with the following $(Q,T)$-affine periodicity
\begin{equation}\label{eap}
\begin{split}
&f(t+T,x)=Qf(t,Q^{-1}x),\\
&g(t+T,x)=Qg(t,Q^{-1}x),
\end{split}
\end{equation}
for some invertible affine matrix $Q\in\mathbb{R}^{l\times l}$, period $T>0$ and all $t\in\mathbb{R}^+$, $x\in\mathbb{R}^l$. When $Q$ is an orthogonal matrix, \eqref{sde} can describe some models such as rigid body rotation under stochastic perturbation.

The probability density $p:\mathbb{R}^+\times\mathbb{R}^l\rightarrow\mathbb{R}^+$ for the solution $X$ of equation \eqref{sde} satisfies a Fokker-Planck equation
\begin{equation}\label{fpe}
\frac{\partial}{\partial t}p(t,x)=-\sum_i\frac{\partial}{\partial x_i}\big(f_i(t,x)p(t,x)\big)+\frac{1}{2}\sum_{ij}\frac{\partial^2}{\partial x_i\partial x_j}\big((gg^\top)_{ij}(t,x)p(t,x)\big),
\end{equation}
where $\top$ represents transpose. Moreover, this probability density is called a $(Q,T)$-affine periodic solution of equation \eqref{fpe} provided that for all $(t,x)\in\mathbb{R}^+\times\mathbb{R}^l$,
\begin{equation}\label{introdefp}
p(t+T,x)=p(t,Q^{-1}x).
\end{equation}
The solution $X$ of \eqref{sde} is called a $(Q,T)$-affine periodic solution if its probability density satisfies condition \eqref{introdefp}. In present paper, we touch affine periodic phenomena under stochastic perturbation.
 
Periodic phenomena have already been observed since Kepler and Newton's time when they studied orbits of planets in solar system. Yet, the definition of periodic solutions was first introduced by H. Poincar\'e in his famous works \cite{P1,P2,P3}. Since then, the existence of periodic solutions is the central topic of dynamical systems all along. For a deterministic periodic differential equation (i.e. $Q=I_l$ where $I_l$ is an $l$-dimensional identity matrix and $g\equiv 0$), namely,
\begin{equation}\label{ode}
x'(t)=f(t,x(t)),
\end{equation}
Massera \cite{M} showed that for $2$-dimensional ordinary differential equations, the boundedness of solutions is sufficient to obtain the existence of periodic solutions. As is well known, the boundedness of solutions fails the sufficiency for higher dimensions until Halanay posed the following criterion \cite{H}
\[
\lim_{t\rightarrow+\infty}(x(t+T)-x(t))=0,
\]
which leads to the existence of periodic solutions. Under some general hypotheses, this condition can be changed to another criterion by LaSalle stationary oscillation principle \cite{LL}. That is, there exists a continuous function $a:\mathbb{R}^+\rightarrow\mathbb{R}^+\backslash\{0\}$ satisfying $\lim\limits_{t\rightarrow\infty}$a(t)=0, such that
\[
|x(t)-y(t)|\leq a(t)|x_0-y_0|
\]
for all $t\geq 0$, where $x$ and $y$ are solutions of \eqref{ode} started from $x_0$ and $y_0$ respectively.

The development of thermodynamics in the 20th century contributes to the research of stochastic dynamical systems. After A. Einstein, A. D. Fokker, M. Planck and A. Kolmogorov's works \cite{E,F,PM,K1}, Fokker-Planck equation is applied to describe the distribution state of diffusion. Based on the construction of random walk and Wiener process, stochastic calculus is established by It\^o \cite{IK}, Stratonovich \cite{SRL} and Maliavin \cite{MP} et al.. The theory of stochastic differential equations has been concerned gradually.

It is a natural question whether the system admits a periodic solution after a random perturbation. Among all others, recurrence is a respectively rough concept to describe the ``periodicity'' of stochastic processes. The definition of recurrence for stochastic processes was introduced by A. Kolmogorov \cite{K2}, which concerns the probability and the frequency of that the diffusion particle returns to a certain state. Limited by the lack of analysis tools, numerous works, such as \cite{HL, HJLY1,HJLY2, HJLY3, HJLY4,BKR}, are devoted to recurrence and stationary measure. 

However, it is more difficult to study other better recurrence, e.g. stochastic periodicity. Although there has been some works (we refer to \cite{CL1,KR,CL2,FWZ}), this topic is still thought to be  a challenging problem. On one hand, for any fixed time, the point on the orbit of the solution in deterministic system becomes a random variable in stochastic system. Point-wise equivalence of two adjacent states between a period cannot display the diffusion of the process. This makes us to consider some macro periodicity, like in the sense of distribution. On the other hand, the space of stochastic processes is non compact. It is necessary to find effective analysis tools. So far, besides Banach's fixed point theorem, almost all other fixed point theorems do not work on stochastic periodicity.

Recently, some breakthroughs have been made on this issue. Z. Liu and W. Wang \cite {LW} obtained the existence of stochastic almost periodic solutions in distribution by Favard separation method. F. Chen el at. \cite{CH} surmounted the difficulty of convergence by Skorokhod theorems \cite{Sav} and gave a weak Halanay type criterion to get the existence of periodic solutions for stochastic differential equations in the sense of distribution. This criterion can be regarded as a law of large numbers for periodic stochastic systems. They also obtained the existence and asymptotic stability of periodic solutions in distributions via Lyapunov's method. M. Ji el at. \cite{JQ} studied existence of periodic solutions in distribution for stochastic differential equations with irregular coefficients.

For deterministic system \eqref{ode}, another special kind of recurrence, affine periodicity, is stepping into the field of researchers. In such systems, $f$ satisfying \eqref{eap} is an affine periodic function, where geometry similarity is taken into account. The affine periodic systems contain several special cases, such as periodic systems ($Q=I_l$), anti-periodic systems ($Q=-I_l$) and rotation periodic systems ($Q\in\mathcal{O}(l)$), which are discussed in many literatures like \cite{WCL,HH,MJ,CAY}. For general affine matrix $Q$, Y. Li et al. \cite{CL,LH,ML,WYL,WYLL,XYL} obtained the existence of affine-periodic solutions for several kinds of deterministic affine-periodic systems. 

Similar to the periodic case, it is also interesting to checkout the preservation of the affine-periodicity for deterministic affine-periodic systems after a random perturbation. The present paper is devoted to obtaining the existence of stochastic affine-periodic solutions for equation \eqref{sde}. The definition of such solutions is in the sense of distribution. Unlike the convergence results in Skorokhod theorems, we also adopt Theorem 3.1 in \cite{LW} which displays uniform convergence in distribution of solutions for a sequence of stochastic differential equations. Applying the techniques of Prohorov's theorem \cite{P} and Arzela-Ascoli theorem (Theorem 7.17 in \cite{K}), the convergence is uniform on any compact intervals of $\mathbb{R}^+$, where the uniformly boundedness in mean square of the solutions is necessary. Since that the construction of the solution in our paper can be concentrated on $[0,T]$, the uniformly boundedness on this interval is implied by the boundedness of initial state. Thus, we rewrite this theorem in a local version. By this convergence result, we get the existence of stochastic affine-periodic solutions in distribution where the boundedness condition in \cite{CH} is weakened and a weaker Halanay type criterion is given. We call it a law of large numbers on periodicity. We also get a stochastic version of LaSalle's stationary oscillation principle by constructing Poincar\'e map for probability density functions of the solutions of \eqref{sde}. As applications, the existence and asymptotic stability of stochastic affine-periodic solutions are obtained via Lyapunov's method \cite{LAM}, which is a classical technique in dynamical system. Here, the stability is in the sense of mean square and distribution. Now we can conclude that some basic and effective criteria have been established for the existence of stochastic affine periodic solutions in distribution.

The rest of this paper is organized as follows. In Section \ref{pre}, we give some preliminaries. We introduce the notations in our paper and the definitions of related concepts. Some results of convergence are also stated in this section. In Section \ref{exi}, we first list some reasonable hypotheses. We show the existence and asymptotic stability of stochastic affine-periodic solutions in distribution under these conditions, which are law of large numbers on affine periodicity and LaSalle's type principle. These are our main results in this paper. In Section \ref{app}, we illustrate our results by some applications via Lyapunov's method and an example.

\section{Preliminary}\label{pre}

In this section, we give the notations and definitions used in this paper. Some useful results on convergence are also stated.

\subsection{Notations and definitions}

Throughout this paper, we use the following notations.
\begin{longtable}[H]{llp{9cm}l}
\hline
\hspace{0.01in}&\hspace{0.01in}&\hspace{0.01in}&\hspace{0.01in}\\
~&$(\Omega,\mathcal{F},\mathbb{P})$ & probability space with sample space $\Omega$, filtration $\mathcal{F}$ and natural probability measure $\mathbb{P}$&~\\
\hspace{0.01in}&\hspace{0.01in}&\hspace{0.01in}&\hspace{0.01in}\\
 ~&$L^2(\mathbb{P},\mathbb{R}^l)$ & space of all $\mathbb{R}^l$-valued random variables $\xi$ satisfying $\mathbb{E}|\xi|^2<\infty$&~\\
\hspace{0.01in}&\hspace{0.01in}&\hspace{0.01in}&\hspace{0.01in}\\
~&$\mathcal{P}(\mathbb{R}^l)$&set of Borel probability measures on $\mathbb{R}^l$&~\\
\hspace{0.01in}&\hspace{0.01in}&\hspace{0.01in}&\hspace{0.01in}\\
~&$\mathcal{P}^2(\mathbb{R}^l)$&set of Borel probability measures of random variables in $L^2(\mathbb{P},\mathbb{R}^l)$&~\\
\hspace{0.01in}&\hspace{0.01in}&\hspace{0.01in}&\hspace{0.01in}\\
~&$X_\xi$&solution of \eqref{sde} with initial state $\xi$ which is omitted if there is no ambiguity in context&~\\
\hspace{0.01in}&\hspace{0.01in}&\hspace{0.01in}&\hspace{0.01in}\\
~&$p_\xi$&probability density of random variable $\xi$&~\\
\hspace{0.01in}&\hspace{0.01in}&\hspace{0.01in}&\hspace{0.01in}\\
~&$p_{X_\xi(t)}$&probability density of stochastic process $X_\xi$ at $t$&~\\
\hspace{0.01in}&\hspace{0.01in}&\hspace{0.01in}&\hspace{0.01in}\\
~&$\widetilde{~\bullet~}$&statistics in probability space $(\tilde{\Omega},\tilde{\mathcal{F}},\tilde{\mathbb{P}})$ different from $(\Omega,\mathcal{F},\mathbb{P})$&~\\
\hspace{0.01in}&\hspace{0.01in}&\hspace{0.01in}&\hspace{0.01in}\\
~&$I_l$&$l$-dimensional unit matrix&~\\
\hspace{0.01in}&\hspace{0.01in}&\hspace{0.01in}&\hspace{0.01in}\\
~&$A^\top$&transpose of matrix $A$&~\\
\hspace{0.01in}&\hspace{0.01in}&\hspace{0.01in}&\hspace{0.01in}\\
 ~&$\trace A$&trace of matrix $A$&~\\
\hspace{0.01in}&\hspace{0.01in}&\hspace{0.01in}&\hspace{0.01in}\\
 ~&$\hess h$&Hessian matrix of function $h$&~\\
\hspace{0.01in}&\hspace{0.01in}&\hspace{0.01in}&\hspace{0.01in}\\
~&$\jac_h$&Jacobian matrix of function $h$ with respect to $x\in\mathbb{R}^l$&~\\
\hspace{0.01in}&\hspace{0.01in}&\hspace{0.01in}&\hspace{0.01in}\\
 ~&$\mathcal{O}(l)$&set of all $l\times l$ orthogonal matrices&~\\
\hspace{0.01in}&\hspace{0.01in}&\hspace{0.01in}&\hspace{0.01in}\\
\hline
\end{longtable}

We continue to use the definitions in \cite{CH,LW} for the norms of Lipschitz continuous real-valued function $h$ on $\mathbb{R}^l$, distance between measures $\mu$ and $\nu$ in $\mathcal{P}(\mathbb{R}^l)$ and the norm of random variables and stochastic processes,
\begin{align*}
\|h\|_{\infty}&=\sup_{x\in\mathbb{R}^l}|h(x)|,\\
\|h\|_L&=\sup_{x,y\in\mathbb{R}^l,x\neq y}\frac{|h(x)-h(y)|}{|x-y|},\\
\|h\|_{BL}&=\max\{\|h\|_{\infty},\|h\|_L\},\\
d_{BL}(\mu,\nu)&=\sup_{\|h\|_{BL}\leq 1}\left|\int hd(\mu-\nu)\right|,\\
\|\xi\|_2&=\left(\mathbb{E}|\xi|^2\right)^{\frac{1}{2}}=\left(\int_{\Omega}|\xi|^2d\mathbb{P}\right)^{\frac{1}{2}},\\
\|X_\xi\|_{2,I}&=\max_{t\in I}\mathbb{E}|X_\xi(t)|^2,~\text{for}~I\subset\mathbb{R}^+.
\end{align*}
 
Similar to the definitions of periodicity in distribution and stochastic periodic solutions in distribution in \cite{CH}, the definitions of $(Q,T)$-affine periodicity in distribution and stochastic $(Q,T)$-affine periodic solutions in distribution are given as follows for some invertible affine matrix $Q$ and period $T$.

\begin{definition}
An $\mathbb{R}^l$-valued stochastic process $X_\xi$ is said to be $(Q,T)$-affine periodic in distribution if its probability density $p_{X_\xi(\cdot)}$ is $(Q,T)$-affine periodic, namely,
\begin{equation}\label{pqtd}
p_{X_\xi(t+T)}(x)=p_{X_\xi(t)}(Q^{-1}x)~\forall t\in\mathbb{R}^+,~x\in\mathbb{R}^l.
\end{equation}
\end{definition}

Using notation ``$\equivid$'' to represent the equality in distribution of two stochastic processes or random variables, $X_\xi$ is $(Q,T)$-affine periodic in distribution means that
\[
Q^{-1}X_\xi(\cdot+T)\equivid X_\xi(\cdot).
\]

\begin{definition}\label{def_QTAPs}
The solution $X_\xi$ of stochastic differential equation \eqref{sde} with $f$ and $g$ being $(Q,T)$-affine periodic (satisfying \eqref{eap}) is said to be a $(Q,T)$-affine periodic solution in distribution provided that the conditions below hold.
\begin{itemize}
\item[{\bf (C1)}]Stochastic process $X_\xi$ is $(Q,T)$-affine periodic in distribution.
\item[{\bf (C2)}]There exists a stochastic process $\tilde{W}$, which has the same distribution with $W$, such that $Q^{-1}X_\xi(\cdot+T)$ is a solution of stochastic differential equation
    \begin{equation}\label{sde1t}
    dY(t)=f(t,Y(t))dt+g(t,Y(t))d\tilde{W}(t).
    \end{equation}
\end{itemize}
\end{definition}

\begin{definition}\label{assip}
The solution $X_\xi$ of \eqref{sde} is said to be a mean square asymptotically stable $(Q,T)$-affine periodic solution in distribution provided that it is a $(Q,T)$-affine periodic solution of \eqref{sde} in distribution and
\begin{itemize}
\item[{\bf(C3)}]$X_\xi$ is asymptotically stable in mean square. That is,
\[
\mathbb{E}|X_\xi(t)-X_\eta(t)|^2\rightarrow0~as~t\rightarrow\infty,
\]
where $\eta$ is any other initial state.
\end{itemize}
\end{definition}

\begin{definition}\label{assid}
The solution $X_\xi$ of \eqref{sde} is said to be an asymptotically stable $(Q,T)$-affine periodic solution in distribution provided that it is a $(Q,T)$-affine periodic solution of \eqref{sde} in distribution and
\begin{itemize}
\item[{\bf(C4)}]$X_\xi$ is asymptotically stable in distribution. That is,
\[
d_{BL}\left(p_{X_\xi(t)},p_{X_\eta(t)}\right)\rightarrow 0~as~t\rightarrow\infty,
\]
where $\eta$ is any other initial state.
\end{itemize}
\end{definition}

We also need to use the definitions below in this paper.
\begin{definition}
A sequence of measures $\{\mu_n\}\subset\mathcal{P}(\mathbb{R}^l)$ is said to be weakly convergent to a measure $\mu$ provided that for every continuous bounded function $q$ on $\mathbb{R}^l$,
\[
\int_{\mathbb{R}^l}q(x)\mu_n(dx)\rightarrow\int_{\mathbb{R}^l}q(x)\mu(dx),~n\rightarrow\infty.
\]
\end{definition}

\begin{definition}
A sequence of $\mathbb{R}^l$-valued stochastic processes $\{Y_n\}$ is said to be convergent in distribution to a $\mathbb{R}^l$-valued stochastic process $Y$ if for all $t\in\mathbb{R}^+$, $\{p_{Y_n(t)}\}$ is weakly convergent to $p_{Y(t)}$.
\end{definition}

If a sequence of measures $\{\mu_n\}\subset\mathcal{P}(\mathbb{R}^l)$ is weakly convergent to a measure $\mu$ and a sequence of stochastic processes $\{Y_n\}$ is convergent in distribution to a stochastic process $Y$, we use the following notations.
\[
\begin{split}
&\mu_n\wconv\mu,\\
&Y_n\convid Y.
\end{split}
\]

\subsection{Convergence of random variables and stochastic processes}

The results of weakly compactness for stochastic processes are also important in our discussion. Two basic results which were given by Skorokhod \cite{Sav} are stated here.

\begin{lemma}\label{npslimit}
Let $\{Y_n\}_{n=1}^{\infty}$ be a sequence of stochastic processes in $\mathbb{R}^l$ such that for every $\{t_i\}_{i=1}^k\subset\mathbb{R}^+$ the joint distribution of $\{Y_n(t_i)\}_{i=1}^k$ is weakly convergent as $n\rightarrow\infty$ and the sequence of $\{Y_n\}_{n=1}^{\infty}$ is uniformly stochastic continuous, that is,
\begin{equation}\label{pusc}
\sup_{n,|s_1-s_2|<\Delta t}\mathbb{P}\{|Y_n(s_1)-Y_n(s_2)|>\varepsilon\}\rightarrow 0,~\Delta t\rightarrow 0.
\end{equation}
Then a sequence of stochastic processes $\{\tilde{Y}_n\}_{n=1}^{\infty}$ can be constructed in another probability space $(\tilde{\Omega},\tilde{\mathcal{F}},\tilde{\mathbb{P}})$ such that
\begin{itemize}
\item[(i)]$\tilde{Y}_n\convid \tilde{Y}$.
\item[(ii)]The finite-dimensional distributions of the processes $Y_n$ and $\tilde{Y}_n$ coincide for $n>0$.
\end{itemize}
\end{lemma}

\begin{lemma}\label{subconv}
Assume that the sequence of stochastic processes $\{Y_n\}_{n=1}^{\infty}$ is uniformly stochastic continuous (satisfying condition \eqref{pusc}) and uniformly bounded in probability, that is,
\begin{equation}\label{pub}
\sup_{t,n}\mathbb{P}\{Y_n(t)>M\}\rightarrow 0,~M\rightarrow\infty.
\end{equation}
Then $\{Y_n\}_{n=1}^{\infty}$ contains a subsequence $\{Y_{n_k}\}_{k=1}^{\infty}$ with weakly convergent finite-dimensional distributions.
\end{lemma}

We also need a local version of Theorem 3.1 in \cite{LW}, which is stated as a following theorem.


\begin{theorem}\label{cor}
Consider the following stochastic differential equations
\[
dX_{\xi_n}=f(t,X_{\xi_n})dt+g(t,X_{\xi_n})dW_n,
\]
where $f,~g$ are as outlined above satisfying linear growth condition and global Lipschitz condition. $\{W_n\}$ are identical distributed Brownian motions. ${\xi_n}$ is a sequence of random variables in $L^2(\mathbb{P},\mathbb{R}^l)$. Moreover, assume that there is a common $r_0>0$ such that for all $n=1,2,\ldots$, $\|X_{\xi_n}\|_{2,[0,T]}\leq r_0$, where $\|\cdot\|_{2,I}$ is defined as above. Then there exists a subsequence of $\{X_{\xi_n}\}$ (still denoted by $\{X_{\xi_n}\}$) such that
\[
X_{\xi_n}\convid X
\]
uniformly on $[0,T]$. Further, there is a Brownian motion $W$ identical distributed with $\{W_n\}$ such that $(X,W)$ is a solution of \eqref{sde} satisfying $\|X\|_{2,[0,T]}\leq r_0$.
\end{theorem}
\begin{proof}
Linear growth condition and global Lipschitz condition lead to the well-posedness of the strong solution for \eqref{sde}, which implies the weak uniqueness. Thus, $X_{\xi_n}$ shares the same distribution with stochastic process $Y_n$, which is the solution of the following equation
\[
\left\{
\begin{split}
&dY_n=f(t,Y_n)dt+g(t,Y_n)dB,\\
&Y_n(0)=\xi_n,~a.s.,
\end{split}
\right.
\]
where $B$ is identical distributed with $\{W_n\}$.

Hence, the result is obtained by following the proof of Theorem 3.1 on $[0,T]$ in \cite{LW}.
\end{proof}

Finally, a Gronwall-type inequality is cited below which is useful in our example.

\begin{lemma}\label{gron}
{\bf(Theorem 11 (i) in \cite{DS})} Suppose $x$ and $K$ are continuous. Functions $A$ and $B$ are Riemann integrable functions on $[t_1,t_2]$ with $B$ and $K$ nonnegative. If
\[
x(t)\leq A(t)+B(t)\int^t_{t_1}K(s)x(s)ds,~t\in[t_1,t_2],
\]
then
\[
x(t)\leq A(t)+B(t)\int^t_{t_1}A(s)K(s)\exp\left(\int^t_sB(r)K(r)dr\right)ds,~t\in[t_1,t_2]
\]
\end{lemma}

\section{A law of large numbers and LaSalle principle for $(Q,T)$-affine periodic solutions}\label{exi}

\subsection{Hypotheses}

For stochastic differential equation \eqref{sde} and the diffusion process $X_\xi$ satisfying this equation, some reasonable assumptions are stated below.
\begin{itemize}
\item[{\bf (H1)}]The drift coefficient $f$ and diffusion coefficient $g$ satisfy the linear growth condition. Namely, there exists a positive constant $G$ such that
    \begin{equation}\label{linearg}
    |f(t,x)|+|g(t,x)|\leq G(1+|x|)
    \end{equation}
    for all $(t,x)\in\mathbb{R}^+\times\mathbb{R}^l$.
\item[{\bf (H2)}]The drift coefficient $f$ and diffusion coefficient $g$ satisfy the $Lipschitz$ condition
    \begin{equation}\label{Lip}
    |f(t,x_1)-f(t,x_2)|\vee|g(t,x_1)-g(t,x_2)|\leq L|x_1-x_2|,
    \end{equation}
    where $L>0$ is a constant, $t\geq 0,~x_i\in\mathbb{R}^l,~i=1,~2.$
\item[{\bf (H3)}]$X_\xi$ satisfies a $(Q,T)$-affine periodic boundedness condition
    \begin{equation}\label{adbc}
    \mathbb{E}|Q^{-n}X_\xi(nT)|^2\leq C
    \end{equation}
    for all $n\geq 0$, where $C$ is a positive constant independent of $n$.
\item[{\bf (H4)}]For the stochastic processes $\{Q^{-n}X_\xi\}_{n=0}^{\infty}$, the probability density $p_{Q^{-n}X_\xi(\cdot)}$ with respect to each $Q^{-n}X$ for $n=1,2,\ldots$ satisfies condition
\begin{equation}\label{waHc}
\lim_{k\rightarrow\infty}\frac{1}{n_k+1}\sum_{N=0}^{n_k}d_{BL}\Big(p_{Q^{-N-1}X_\xi(NT+T)},p_{Q^{-N}X_\xi(NT)}\Big)=0,
\end{equation}
where $\{n_k\}$ is a sequence of integers tending to $+\infty$.
\end{itemize}

\begin{remark}
Conditions {\bf (H1)-(H2)} ensure the existence and uniqueness of strong solutions for \eqref{sde} (for more details see Section 5.2-5.3 in \cite{O}).
\end{remark}
\begin{remark}
The dominated boundedness Condition (4) in \cite{CH} that there exists an $L^2$ function $\tilde{L}$ such that 
\begin{equation}\label{db}
|X_\xi(nT)|\leq \tilde{L},~\mathbb{E}|\tilde{L}|^2\leq\infty~\forall n\geq 0
\end{equation}
for $Q=I_l$ is changed to {\bf(H3)}, which is much weaker. In fact, \eqref{db} is too strong to be within reach for most solutions of \eqref{sde}.
\end{remark}
\begin{remark}
{\bf(H4)} is a weak stochastic affine-periodic version of Halanay's criterion. Note that when $g\equiv 0$, \eqref{waHc} becomes
\[
\lim_{n\rightarrow\infty}\frac{1}{n_k+1}\sum^{n_k}_{N=0}\left|Q^{-N-1}X((N+1)T)-Q^{-N}X(NT)\right|=0.
\]
Our result below show that under this affine-periodic averaging condition, the deterministic affine-periodic system admits an affine-periodic solution. \eqref{waHc} can also be regarded as a law of large numbers with geometry structure.
\end{remark}

\begin{remark}\label{h42}
Condition \eqref{waHc} seems difficult to be verified. Since convergence in mean square implies convergence in distribution, we provide a following stronger condition, which can be regarded as a Bernoulli type law of large numbers on stochastic periodicity. 
\begin{itemize}
\item[{\bf(H4)'}]Stochastic processes $\{Q^{-n}X_\xi\}_{n=0}^\infty$ satisfy
\begin{equation}\label{waHc2}
\lim_{k\rightarrow\infty}\frac{1}{n_k+1}\sum^{n_k}_{N=0}\left\|Q^{-N-1}X_{\xi}(NT+T)-Q^{-N}X_\xi(NT)\right\|_2^2=0,
\end{equation}
where $\{n_k\}$ are as above.
\end{itemize}
\end{remark}

\subsection{Existence of stochastic affine-periodic solutions}

Now we can state our first result which is a criterion on the existence of $(Q,T)$-affine periodic solutions in distribution for stochastic differential equation \eqref{sde}. 

\begin{theorem}\label{t1}
Assume that $f$ and $g$ are continuous and $(Q,T)$-affine periodic funtions satisfying assumptions {\bf (H1)}-{\bf (H4)}. Then there exists an $L^2$-bounded $(Q,T)$-affine periodic solution in distribution of \eqref{sde}.
\end{theorem}

\begin{proof}
The well-posedness of \eqref{sde} is insured by {\bf(H1)-(H2)}. Let $X_\eta$ be the solution satisfying the assumptions above. We tend to construct a $(Q,T)$-affine periodic solution in distribution of \eqref{sde}.

{\bf Step 1: Construct possible initial random variable.}

Given a random variable $\xi_k$ independent of $W$ and $\eta$ such that
\[
\mathbb{P}\{\xi_k=N\}=\frac{1}{k+1},~N=0,~1,\ldots,~k,
\]
for each $k\in\mathbb{Z}^+$, we define a sequence of stochastic processes
\[
\left\{
\begin{split}
&Y_k(t)=Q^{-\xi_k}X_\eta(t+\xi_kT),\\
&Y_k(0)=Q^{-\xi_k}X_\eta(\xi_kT).
\end{split}
\right.
\]

For $\xi_k=N$, $Y_k$ is a solution of stochastic differential function
\begin{equation}\label{sdent}
dY(t)=f(t,Y(t))dt+g(t,Y(t))dW_N(t),
\end{equation}
where the law of stochastic process
\[
W_N(t)=W(t+NT)-W(NT),
\]
coincides with the law of $W(t)$. Moreover, one can verify that
\begin{equation}\label{rvBW}
W_{\xi_k}\equivid W,
\end{equation}
by replacing $N$ to $\xi_k$ for any fixed $k=0,~1,~\ldots$. Namely, $\{W_{\xi_k}\}$ are different Brownian motions in $(\Omega,\mathcal{F},\mathbb{P})$ and $\{(Y_k,W_{\xi_k})\}$ are solutions of \eqref{sde}.

Because of the independence between $\xi_k$ and $W$, we have
\begin{equation}\label{indpxw}
\mathbb{P}\{Y_k(t)\in A\}=\frac{1}{k+1}\sum^k_{N=0}\mathbb{P}\{Q^{-N}X_\eta(t+NT)\in A\}
\end{equation}
for any Borel set $A\subset\mathbb{R}^l$. Thus by Chebyshev's inequality and affine-periodic boundedness condition {\bf (H3)}, we have
\[
\begin{split}
\mathbb{P}\{|Y_k(0)|>M\}&=\frac{1}{k+1}\sum^k_{N=0}\mathbb{P}\{|Q^{-N}X_\eta(NT)|>M\}\\
                        &\leq\frac{1}{k+1}\sum^k_{N=0}\frac{\mathbb{E}|Q^{-N}X_\eta(NT)|^2}{M^2}\\
                        &<\frac{C}{M^2}\rightarrow 0,~M\rightarrow\infty.
\end{split}
\]

Taking $Y_k(0)$ as a stochastic process (constant in time), \eqref{pub} is fulfilled with the estimation above and \eqref{pusc} is trivial. By Lemmas \ref{npslimit} and \ref{subconv}, there is another probability space $(\tilde{\Omega},\tilde{\mathcal{F}},\tilde{\mathbb{P}})$ such that
\begin{itemize}
\item there exists a sequence $\{\tilde{Y}_{k,0}\}$ in this space satisfying $\tilde{Y}_{k,0}\equivid Y_k(0)$ with respect to each $k=0,~1,~\ldots$.
\item This new sequence contains a subsequence $\{\tilde{Y}_{n_k,0}\}$ satisfying $\tilde{Y}_{n_k,0}\convid\tilde{Y}_0$ as $k\rightarrow\infty$, where $\tilde{Y}_0$ is a random variable in $(\tilde{\Omega},\tilde{\mathcal{F}},\tilde{\mathbb{P}})$.
\end{itemize}
Thus,
\begin{equation}\label{initst}
\mathbb{E}|Y_{n_k,0}|^2=\tilde{\mathbb{E}}|\tilde{Y}_{n_k,0}|^2\leq C.
\end{equation}

{\bf Step 2: Verify that there is a subsequence of $Y_{n_k}$ convergent in distribution on $[0,T]$.}

By Cauchy-Schwartz inequality, It\^o isometry and linear growth condition \eqref{linearg}, we have for all $k\in\mathbb{N}^+$,
\begin{align*}
\mathbb{E}|Y_{n_k}(t)|^2=&\mathbb{E}\left|Y_{n_k}(0)+\int^t_0f(s,Y_{n_k}(s))ds+\int^t_0g(s,Y_{n_k}(s))dW_{\xi_{n_k}}(s)\right|^2\\
                         \leq&3\mathbb{E}|Y_{n_k}(0)|^2+3\mathbb{E}\left|\int^t_0f(s,Y_{n_k}(s))ds\right|^2\\
&+3\mathbb{E}\left|\int^t_0g(s,Y_{n_k}(s))dW_{\xi_{n_k}}(s)\right|^2\\
\leq&3\mathbb{E}|Y_{n_k}(0)|^2+3t\mathbb{E}\int^t_0|f(s,Y_{n_k}(s))|^2ds\\
&+3\mathbb{E}\int^t_0|g(s,Y_{n_k}(s))|^2ds\\
\leq&3\mathbb{E}|Y_{n_k}(0)|^2+6(t+1)tG^2+6(T+1)G^2\int^t_0\mathbb{E}|Y_{n_k}(s)|^2ds.
\end{align*}
Hence, by Gronwall's inequality and {\bf (H3)}, it holds that
\[
\begin{split}
\mathbb{E}|Y_{n_k}(t)|^2\leq&3e^{6(T+1)tG^2}(\mathbb{E}|Y_{n_k}(0)|^2+2(t+1)tG^2)\\
\leq&3e^{6(T+1)TG^2}(C+2(T+1)TG^2),
\end{split}
\]
uniformly with respect to $k$ and on $[0,T]$. That is, there exists a positive constant $r_0=3e^{6(T+1)TG^2}(C+2(T+1)TG^2)$ independent with $k$ such that
\[
\|Y_{n_k}\|_{2,[0,T]}\leq r_0.
\]

By Theorem \ref{cor}, this at once leads to that there is a subsequence of $\{Y_{n_k}\}$ (still denoted by $\{Y_{n_k}\}$) and a weak solution $(Z,W)$ of \eqref{sde} such that
\[
Y_{n_k}\convid Z,
\]
uniformly on $[0,T]$. Moreover,
\[
\|Z\|_{2,[0,T]}\leq r_0.
\]

{\bf Step 3: Show that $Z$ is the desired solution.}

To show that $Z$ is a $(Q,T)$-affine periodic solution in distribution of \eqref{sde}, we only need to verify that
\[
Q^{-1}Z(T)\equivid Z(0).
\]
By the definitions of $\{Y_k\}_{k=0}^{\infty}$ and assumption {\bf (H4)}, we have
\begin{align*}
d_{BL}&\left(p_{Q^{-1}Z(T)},p_{Z(0)}\right)\\
&=d_{BL}\left(\tilde{p}_{Q^{-1}\tilde{Y}_{0_1}(T)},\tilde{p}_{\tilde{Y}_{0_1}(0)}\right)\\
                                                      &=\lim_{k\rightarrow\infty}d_{BL}\left(\tilde{p}_{Q^{-1}\tilde{Y}_{n_k}(T)},\tilde{p}_{\tilde{Y}_{n_k}(0)}\right)\\
                                                      &=\lim_{k\rightarrow\infty}d_{BL}\left(p_{Q^{-1}Y_{n_k}(T)},p_{Y_{n_k}(0)}\right)\\
                                                      &=\lim_{k\rightarrow\infty}\sup_{\|\phi\|_{BL}\leq 1}\left|\int\phi dp_{Q^{-1}Y_{n_k}(T)}-\int\phi dp_{Y_{n_k}(0)}\right|\\
                                                      &=\lim_{k\rightarrow\infty}\sup_{\|\phi\|_{BL}\leq 1}\left|\int_{\Omega}\left(\phi\left(Q^{-1}Y_{n_k}(T)\right)-\phi\left(Y_{n_k}(0)\right)\right)d\mathbb{P}\right|\\
                                                      &=\lim_{k\rightarrow\infty}\sup_{\|\phi\|_{BL}\leq 1}\Bigg|\int_{\Omega}\Big(\phi\left(Q^{-\xi_{n_k}-1}X_\eta((\xi_{n_k}+1)T)\right)\\
                                                                                 &~~~~~~~~~~~~~~~~~~~~~~~~~~~~-\phi\left(Q^{-\xi_{n_k}}X_\eta(\xi_{n_k}T)\right)\Big)d\mathbb{P}\Bigg|
                                                    \\
                                                      &=\lim_{k\rightarrow\infty}\sup_{\|\phi\|_{BL}\leq 1}\Bigg|\frac{1}{n_k+1}\sum^{n_k}_{N=0}\int_{\Omega}\Big(\phi\left(Q^{-N-1}X_\eta((N+1)T)\right)\\
                                                                                                                &~~~~~~~~~~~~~~~~~~~~~~~~~~~~~~~~~~~~~~~~~~-\phi\left(Q^{-N}X_\eta(NT)\right)\Big)d\mathbb{P}\Bigg|
                                                       \\
                                                      &\leq\lim_{k\rightarrow\infty}\frac{1}{n_k+1}\sum^{n_k}_{N=0}d_{BL}\left(p_{Q^{-N-1}X_\eta(NT+T)},p_{Q^{-N}X_\eta(NT)}\right)\\
                                                      &=0.
\end{align*}
This leads to that $Q^{-1}Z(T)\equivid Z(0)$, which means that the condition {\bf (C1)} in Definition \ref{def_QTAPs} holds.

Furthermore, $(Q^{-1}Z(\cdot+T),\tilde{W})$ is also a solution of \eqref{sde} with $\tilde{W}=W_1,~a.e.$ by the $(Q,T)$-affine periodicity of $f$ and $g$. This leads to the condition {\bf (C2)} in Definition \ref{def_QTAPs}.

Then $Z$ is a $(Q,T)$-affine periodic solution in distribution of equation \eqref{sde}. The proof is completed.
\end{proof}

Changing Condition {\bf(H4)} to {\bf(H4)'}, we also have a result below. Since convergence in mean square implies convergence in distribution, the proof is omitted.
\begin{cor}\label{cort1}
Assume that $f$ and $g$ are continuous and $(Q,T)$-affine periodic functions satisfying assumptions {\bf(H1)-(H3)} and {\bf(H4)'}. Then there exists an $L^2$-bounded $(Q,T)$-affine periodic solution in distribution of \eqref{sde}.
\end{cor}

\subsection{LaSalle type stationary principle for stochastic differential equations}

The rest of this section is devoted to a LaSalle's type principle for stochastic $(Q,T)$-affine periodic systems. First, a useful formula is stated below.

\begin{lemma}\label{lemma}
If $X_\xi$ is a solution of $(Q,T)$-affine periodic system \eqref{sde} and conditions {\bf (H1)-(H2)} are satisfied, then
\begin{equation}\label{poinc}
p_{Q^{-k}X_\xi(t+kT)}=p_{X_{Q^{-k}X_\xi(kT)}(t)}
\end{equation}
for all $k\in\mathbb{N}$ and $\xi\in L^2(\mathbb{P},\mathbb{R}^l)$.
\end{lemma}

\begin{proof}
Since $X$ is a solution of \eqref{sde} starting from $\xi$, we have
\[
\begin{split}
X_\xi(t+kT)&=\xi+\int^{t+kT}_0dX_\xi(t)\\
           &=\xi+\int^{kT}_{0}dX_\xi(t)+\int^{t+kT}_{kT}dX_\xi(t)\\
           &=X_\xi(kT)+\int^{t+kT}_{kT}f(s,X_\xi(s))ds+\int^{t+kT}_{kT}g(s,X_\xi(s))dW(s)\\
           &~\begin{split}
            =&X_\xi(kT)+\int^t_0f(s+kT,X_\xi(s+kT))ds\\
             &+\int^t_0g(s+kT,X_\xi(s+kT))dW_k(s)
            \end{split}\\
           &~\begin{split}
            =X_\xi(kT)+Q^k\Bigg(&\int^t_0f(s,Q^{-k}X_\xi(s+kT))ds\\
            &+\int^t_0g(s,Q^{-k}X_\xi(s+kT))dW_k(s)\Bigg),
            \end{split}
\end{split}
\]
where $\{W_k\}$ are as above.

Hence, $(Q^{-k}X_\xi(\cdot+kT),W_k)$ is a weak solution of \eqref{sde} with initial state $Q^{-k}X_\xi(kT)$. By the uniqueness of weak solutions, we have
\[
p_{Q^{-k}X_{\xi}(t+kT)}=p_{X_{Q^{-k}X_\xi(kT)}(t)}.
\]
\end{proof}

Now we can state our result.

\begin{theorem}\label{tla}
Assume that \eqref{sde} is a $(Q,T)$-affine periodic system where $Q\in\mathcal{O}(n)$. Continuous functions $f$ and $g$ satisfy conditions {\bf(H1)-(H2)}. Moreover, assume that 
\begin{itemize}
\item[\bf(H5)]there is a continuous function $a:\mathbb{R}^+\rightarrow\mathbb{R}^+\backslash\{0\}$ with $r:=\varlimsup_{k\rightarrow\infty}a(kT)<1$ such that for any solutions $X_{\xi_1}$ and $X_{\xi_2}$ of \eqref{sde},
\begin{equation}\label{initcont}
d_{BL}\left(p_{X_{\xi_1}(t)},p_{X_{\xi_2}(t)}\right)\leq a(t)d_{BL}\left(p_{\xi_1},p_{\xi_2}\right),
\end{equation}
where $\xi_i\in L^2(\mathbb{P},\mathbb{R}^l)$ for $i=1,2$. 
\end{itemize}
Then \eqref{sde} has a unique asymptotically stable $(Q,T)$-affine periodic solution in distribution where the uniqueness is in the sense of distribution.
\end{theorem}

\begin{proof}
We still begin with the construction of initial state. The strong well-posedness of \eqref{sde} with any initial state $\xi\in L^2(\mathbb{P},\mathbb{R}^l)$ is ensured by {\bf(H1)-(H2)}. We define this strong solution by $X_\xi$. Now we can introduce the Poincar\'e map $P:\mathcal{P}^2(\mathbb{R}^l)\rightarrow\mathcal{P}^2(\mathbb{R}^l)$,
\[
P(p_\xi)=p_{Q^{-1}X_\xi(T)}~\forall\xi\in L^2(\mathbb{P},\mathbb{R}^l).
\]
By the weak uniqueness of solutions for \eqref{sde}, $P$ is well-defined.

We claim that for all $k\in\mathbb{N}$,
\begin{equation}\label{claim}
P^k(p_\xi)=p_{Q^{-k}X_\xi(kT)}.
\end{equation}
We prove this claim by induction. For $k=0$ and $k=1$, \eqref{claim} is obtained by the definition of $P$. Assume that \eqref{claim} holds for $k-1\in\mathbb{N}$, we still need to show that \eqref{claim} holds for $k$. By the weak uniqueness of the solutions for \eqref{sde} and the induction hypothesis, we have
\[
\begin{split}
P^k(p_\xi)=&P\comp P^{k-1}(p_\xi)\\
           =&P(p_{Q^{-k+1}X_\xi(kT-T)})\\
           =&p_{Q^{-1}X_{Q^{-k+1}X_\xi(kT-T)}(T)}.
\end{split}
\]
By \eqref{poinc} we have
\[
p_{X_{Q^{-k+1}X_\xi(kT-T)}(T)}=p_{Q^{-k+1}X_\xi(kT)}.
\]
Thus,
\[
P^k(p_\xi)=p_{Q^{-k}X_\xi(kT)}.
\]
By induction, \eqref{claim} holds for all $k\in\mathbb{N}$.

By the definition of $r$, there exist a  $k_0\in\mathbb{N}$ and a constant $r_1\in(r,1)$ such that for all $k\geq k_0$,
\[
a(kT)\leq r_1.
\]
Thus, by condition \eqref{initcont} we have
\begin{align*}
&d_{BL}\left(P^k(p_\xi),P^k(p_\lambda)\right)\\
       &~~=d_{BL}\left(p_{Q^{-k}X_\xi(kT)},p_{Q^{-k}X_\lambda(kT)}\right)\\
       &~~=\sup_{\|h\|_{BL}\leq 1}\left|\int hd\left(p_{Q^{-k}X_\xi(kT)}-p_{Q^{-k}X_\lambda(kT)}\right)\right|\\
       &~~=\sup_{\|h\|_{BL}\leq 1}\left|\int_{\mathbb{R}^l}h(x)\left(\mathbb{P}(X_\xi(kT)\in dQ^kx)-\mathbb{P}(X_\lambda(kT)\in dQ^kx)\right)\right|\\
       &~~=\sup_{\|h\|_{BL}\leq 1}\left|\int_{\mathbb{R}^l}h(Q^{-k}y)\left(\mathbb{P}(X_\xi(kT)\in dy)-\mathbb{P}(X_\lambda(kT)\in dy)\right)\right|\\
       &~~=d_{BL}\left(p_{X_\xi(kT)},p_{X_\lambda(kT)}\right)\\
       &~~\leq a(kT)d_{BL}(p_\xi,p_\lambda)\\
       &~~\leq r_1d_{BL}(p_\xi,p_\lambda)
\end{align*}
for all $\xi, \lambda\in L^2(\mathbb{P},\mathbb{R}^l)$. This means that $P$ is an $r_1$-contraction map of order $k_0$. Thus, by contraction mapping fixed point theorem, $P$ has a unique fixed point $p_{*}\in\mathcal{P}^2(\mathbb{R}^l)$. Moreover, we can find a random variable $\xi_0\in L^2(\mathbb{P},\mathbb{R}^l)$ such that $p_{\xi_0}=p_{*}$. Thus, we have that
\[
p_{Q^{-1}X_{\xi_0}(T)}=P(p_{\xi_0})=p_{\xi_0}.
\]
Furthermore, \eqref{claim} leads to that
\[
p_{\xi_0}=P^k(p_{\xi_0})=p_{Q^{-k}X_{\xi_0}(kT)}~\forall k\in\mathbb{N}.
\]
This means that conditions {\bf(H3)} and {\bf(H4)} are fulfilled. Together with {\bf(H1)} and {\bf(H2)}, Equation \eqref{sde} has a unique $(Q,T)$-affine periodic solution in distribution which is $X_{\xi_0}$ by Theorem \ref{t1}.

It remains to show that $X_{\xi_0}$ is asymptotically stable in distribution. Let $\mu\in L^2(\mathbb{P},\mathbb{R}^l)$. Then {\bf(H1)-(H2)} ensure that $X_\mu$ is a strong solution of \eqref{sde}. Now we show that
\begin{equation}\label{t3asid}
d_{BL}\left(p_{X_{\xi_0}(t)},p_{X_\mu(t)}\right)\rightarrow 0~as~t\rightarrow\infty.
\end{equation}
To this end, we only need the following estimate. That is, for all $t\geq 0$ and $m\in\mathbb{N}^+$,
\begin{equation}\label{itres}
d_{BL}\left(p_{X_{\xi_0}(t+mk_0T)},p_{X_\mu(t+mk_0T)}\right)\leq a(t)r_1^md_{BL}(p_{\xi_0},p_\mu).
\end{equation}

By \eqref{poinc} and \eqref{initcont}, we have
\[
\begin{split}
d_{BL}&\left(p_{X_{\xi_0}(t+k_0T)},p_{X_\mu(t+k_0T)}\right)\\
&=d_{BL}\left(p_{Q^{-k_0}X_{\xi_0}(t+k_0T)},p_{Q^{-k_0}X_\mu(t+k_0T)}\right)\\
&=d_{BL}\left(p_{X_{\xi_0}(t+k_0T)},p_{X_\mu(t+k_0T)}\right)\\
&\leq a(t)d_{BL}\left(p_{X_{\xi_0}(k_0T)},p_{X_\mu(k_0T)}\right)\\
&\leq a(t)r_1d_{BL}\left(p_{\xi_0},p_\mu\right).
\end{split}
\]
Hence, \eqref{itres} holds for $m=1$ and any $t\in\mathbb{R}^+$.

Repeating this estimation, we have
\begin{align*}
d_{BL}&\left(p_{X_{\xi_0}(t+mk_0T)},p_{X_\mu(t+mk_0T)}\right)\\
&=d_{BL}\left(p_{X_{Q^{-mk_0}X_{\xi_0}(mk_0T)}(t)},p_{X_{Q^{-mk_0}X_\mu(mk_0T)}(t)}\right)\\
&\leq a(t)d_{BL}\left(p_{X_{\xi_0}(mk_0T)},p_{X_\mu(mk_0T)}\right)\\
&=a(t)d_{BL}\left(p_{Q^{-(m-1)k_0}X_{\xi_0}(mk_0T)},p_{Q^{-(m-1)k_0}X_\mu(mk_0T)}\right)\\
&=a(t)d_{BL}\Big(p_{X_{Q^{-(m-1)k_0T}X_{\xi_0}((m-1)k_0T)}(k_0T)},\\
&~~~~~~~~~~~~~~~~~~~p_{X_{Q^{-(m-1)k_0T}X_\mu((m-1)k_0T)}(k_0T)}\Big)\\
&\leq a(t)a(k_0T)d_{BL}\left(p_{X_{\xi_0}((m-1)k_0T)},p_{X_\mu((m-1)k_0T)}\right)\\
&\leq a(t)r_1a(k_0T)d_{BL}\left(p_{X_{\xi_0}((m-2)k_0T)},p_{X_\mu((m-2)k_0T)}\right)\\
&\cdots\\
&\leq a(t)r_1^{m-1}\cdot a(k_0T)d_{BL}\left(p_{\xi_0},p_\mu\right)\\
&\leq a(t)r_1^md_{BL}\left(p_{\xi_0},p_\mu\right).
\end{align*}
Thus, \eqref{itres} holds for $m$ and  $m$ can be arbitrarily chosen in $\mathbb{N}^+$.

Now let $t\in[0,T]$. The continuity of $a$ leads to the boundedness on $[0,T]$. Thus, by the arbitrariness of $m$ and $r_1\in(r,1)$, $X_{\xi_0}$ is asymptotically stable in distribution. This completes the proof.

\end{proof}

As in Remark \ref{h42}, we also provide a mean square condition with respect to {\bf(H5)}. Furthermore, we strengthen the restriction of function $a$. That is,
\begin{itemize}
\item[{\bf(H5)'}]there is a continuous function $a:\mathbb{R}^+\rightarrow\mathbb{R}\backslash\{0\}$ with
\[
\lim_{k\rightarrow\infty}a(kT)=0\]
such that for any solutions $X_{\xi_1}$ and $X_{\xi_2}$ of \eqref{sde},
\begin{equation}
\left\|X_{\xi_1}(t)-X_{\xi_2}(t)\right\|_2^2\leq a(t)\|\xi_1-\xi_2\|^2_2,
\end{equation}
where $\xi_1$ and $\xi_2$ are as in {\bf (H5)}.
\end{itemize}

Different from Corollary \ref{cort1}, it is not so trivial as it seems to obtain {\bf(H5)} from {\bf(H5)'}. However, {\bf(H5)'} implies {\bf(H3)-(H4)} by compressibility in mean square. The proof is based on relations between different kinds of convergence of random variables, which we would not repeat here. For more details, we refer to theorems in Section 2 of \cite{V}, where rigorous proofs are displayed.

\begin{cor}\label{cort2}
Assume that \eqref{sde} is a $(Q,T)$-affine periodic system where $Q\in\mathcal{O}(n)$. Continuous functions $f$ and $g$ satisfy conditions {\bf (H1)-(H2)} and {\bf(H5)'}. Then \eqref{sde} has a unique asymptotically stable $(Q,T)$-affine periodic solution in distribution where the uniqueness is in the sense of  distribution.
\end{cor}




\section{Applications}\label{app}

\subsection{Existence and asymptotic stability of stochastic affine-periodic solutions via Lyapunov's method}

Similar to the discussion in \cite{CH}, Lyapunov's method is also available in case that $Q\in\mathcal{O}(n)$. In fact, a criterion on existence of mean square asymptotically stable $(Q,T)$-affine periodic solutions of \eqref{sde} is given as follows.

\begin{theorem}\label{t2}
Assume that $f$ and $g$ are $(Q,T)$-affine periodic satisfying {\bf (H1)} and {\bf (H2)}. Moreover, assume that $Q\in\mathcal{O}(l)$. Suppose that there exists a Lyapunov function $V:\mathbb{R}^+\times\mathbb{R}^l\rightarrow\mathbb{R}$ satisfying the following conditions.
\begin{itemize}
\item[{\bf (H6)}]There exist positive constants $A$ and $B$ such that,
    \[
    A|x|^2\leq V(t,x)\leq B|x|^2~~\forall x\in\mathbb{R}^l.
    \]
\item[{\bf (H7)}]There exists a function $\alpha:\mathbb{R}^+\rightarrow\mathbb{R}$ locally integrable and so that $t\rightarrow\infty$ implies $\int^{t}_0\alpha(s)ds\rightarrow-\infty$  as $t\rightarrow\infty$ such that
    \[
    \mathcal{L}V(t,x-y)\leq\alpha(t)V(t,x-y)~\forall(t,x,y)\in\mathbb{R}^+\times\mathbb{R}^l\times\mathbb{R}^l,
    \]
    where
    \begin{equation}\label{lyaop}
    \begin{split}
    \mathcal{L}V&(t,x-y)\\
              :=&\frac{\partial V}{\partial t}(t,x-y)+\left\langle f(t,x)-f(t,y),\frac{\partial V}{\partial x}(t,x-y)\right\rangle\\
                +&\frac{1}{2}\trace\left[\left(g(t,x)-g(t,y)\right)^\top\hess V(t,x-y)\left(g(t,x)-g(t,y)\right)\right].
    \end{split}
    \end{equation}
\end{itemize}
Then there exists a unique  mean square asymptotically stable $(Q,T)$-affine periodic solution in distribution of \eqref{sde}.
\end{theorem}

\begin{proof}
Firstly, we show that there exists a unique $(Q,T)$-affine periodic solution of \eqref{sde}. Let $X_\xi$ be a solution of \eqref{sde} with initial value $\xi\in L^2(\mathbb{P},\mathbb{R}^l)$. Let $X_k(t)=Q^{-k}X_\xi(t+kT)$ a.s. where $k\in\mathbb{N}$. From \eqref{sdent} we know that there is a probability space $(\tilde{\Omega},\tilde{\mathcal{F}},\tilde{\mathbb{P}})$ and a Brownian motion $\tilde{W}$ such that $\{X_k\}$ are the solutions of the following stochastic differential equations
\[
\left\{
\begin{split}
&dX_k=f(t,X_k)dt+g(t,X_k)d\tilde{W},\\
&X_k(0)=Q^{-k}\xi.
\end{split}
\right.
\]

Define an auxiliary stochastic process $Y$ as follows.
\[
Y(t)=e^{-\int^t_0\alpha(s)ds}V(t,X_1(t)-X_\xi(t)).
\]
By It\^o formula, we have
\[
\begin{split}
dY(t)=&-\alpha(t)Y(t)dt+e^{-\int^t_0\alpha(s)ds}\mathcal{L}V(t,X_1(t)-X_\xi(t))dt\\
      &+\left\langle e^{-\int^t_0\alpha(s)ds}\frac{\partial V}{\partial x}(t,X_1(t)-X_\xi(t)),\Big(g(t,X_1(t))-g(t,X_\xi(t))\Big)d\tilde{W}\right\rangle.
\end{split}
\]
Thus assumption {\bf (H7)} leads to that
\[
\begin{split}
\mathbb{E}Y(t)=&~\mathbb{E}\left(Y(0)+\int^t_0dY(s)\right)\\
              =&~\mathbb{E}V(0,X_1(0)-\xi)\\
               &~\begin{split}
                 +\mathbb{E}\int^t_0e^{-\int^s_0\alpha(\tau)d\tau}\Big(&-\alpha(s)V(s,X_1(s)-X_\xi(s))\\
                                                                 &+\mathcal{L}V(s,X_1(s)-X_\xi(s))\Big)ds
                 \end{split}\\
           \leq&~\mathbb{E}V(0,X_1(0)-\xi).
\end{split}
\]

Note that
\[
\begin{split}
X_{k+1}(t)&=Q^{-k-1}X_\xi(t+(k+1)T)~a.s.\\
          &=Q^{-k}Q^{-1}X_\xi(t+kT+T)~a.s.\\
          &=Q^{-k}X_1(t+kT)~a.s..
\end{split}
\]
Hence assumptions {\bf(H6)} and {\bf(H7)} imply that
\[
\begin{split}
\mathbb{E}V&(t+kT,Q^k(X_{k+1}(t)-X_k(t)))\\
          =&~\mathbb{E}V(t+kT,X_1(t+kT)-X_\xi(t+kT))\\
       \leq&~e^{\int^{t+kT}_0\alpha(s)ds}\mathbb{E}V(0,X_1(0)-\xi)\\
       \leq&~e^{\int^{t+kT}_0\alpha(s)ds}\cdot B\mathbb{E}|X_1(0)-\xi|^2\rightarrow0~as~k\rightarrow\infty.
\end{split}
\]
That is, for any $\epsilon>0$, there is a $k_0=k_0(\epsilon)\in\mathbb{N}$ such that $k\geq k_0$ implies that
\begin{equation}\label{left}
\begin{split}
A\mathbb{E}|X_{k+1}(t)-X_k(t)|^2&=\mathbb{E}A|Q^k(X_{k+1}(t)-X_k(t))|^2\\
                                             &\leq\mathbb{E}V(t+kT,Q^k(X_{k+1}(t)-X_k(t)))\\
                                             &\leq e^{\int^{t+kT}_0\alpha(s)ds}\cdot B\mathbb{E}|X_1(0)-\xi|^2\\
                                             &\leq A\epsilon,
\end{split}
\end{equation}
uniformly on $t\in\mathbb{R}^+$.

By {\bf(H6)}, we have
\[
\mathbb{E}|Q^{-k-1}X_\xi(t+(k+1)T)-Q^{-k}X_\xi(t+kT)|^2\rightarrow0~as~k\rightarrow\infty
\]
for all $t\geq 0$. This shows that {\bf(H3)} holds.

Moreover, $\{Q^{-K-1}X_\xi(t+(k+1)T)-Q^{-k}X_\xi(t+kT)\}$ is convergent to $0$ in mean square, which implies the convergence of this sequence in probability and distribution by Chebyshev's inequality. Namely, {\bf(H4)} is also satisfied. Thus, there exists a $(Q,T)$-affine periodic solution $Z$ of \eqref{sde} in distribution by Theorem \ref{t1}.

Next we show that $Z$ is asymptotically stable in mean square as $t\rightarrow\infty$, which also leads to the uniqueness of $Z$. Repeating the estimate above, for any other solution $Z_1$ of \eqref{sde}, we have for any given $\epsilon>0$,
\begin{equation}\label{asyest}
\begin{split}
A\mathbb{E}|Z(t)-Z_1(t)|^2&\leq e^{\int^{t}_0\alpha(s)ds}B\mathbb{E}|Z(0)-Z_1(0)|^2\\
                            &\leq A\epsilon,
\end{split}
\end{equation}
when $t$ is large enough. Thus,
\[
\mathbb{E}|Z(t)-Z_1(t)|^2\rightarrow 0~as~t\rightarrow\infty.
\]
This completes the proof of Theorem \ref{t2}.
\end{proof}

\begin{remark}
Condition {\bf(H6)} is a little bit different from Condition (i) of Theorem 1.3 in \cite{CH}. In fact, function $a$ in \cite{CH} should be required as a convex function, which keeps the direction of inequalities in \eqref{left}. We correct this error here. To simplify the requirement of Lyapunov functions, linear functions $A$ and $B$ are adopted in our theorem.
\end{remark}

For Theorem \ref{tla}, we also have following result.

\begin{theorem}
Assume that \eqref{sde} is a $(Q,T)$-affine periodic system with conditions {\bf(H1)-(H2)} satisfied and that $f$ and $g$ are smooth enough with respect to $x$. Moreover, assume that 
\begin{itemize}
\item[{\bf(H8)}]there are $C^1$ real positive definite matrix-valued function $D: \mathbb{R}^+\rightarrow\mathbb{R}^{l\times l}$ with $D=D^\top$ and locally integrable bounded from above function $\alpha:\mathbb{R}^+\rightarrow\mathbb{R}$ with $\int^{kT}_0\alpha(s)ds\rightarrow-\infty$ as $k\rightarrow\infty$, such that 
\begin{align}
&D'(t)+2\jac_f^\top(t,x)D(t)+\sum^m_{i=1}\jac_{g_i}^\top(t,x)D(t)\jac_{g_i}(t,y)<\alpha(t)D(t),\\
&D(kT)=D(0)~~\forall~~k\in\mathbb{N}
\end{align}
for any $t\in\mathbb{R}^+$ and $x,y\in\mathbb{R}^l$. 
\end{itemize}
Then \eqref{sde} admits a unique asymptotically stable $(Q,T)$-affine periodic solution in distribution.
\end{theorem}

\begin{proof}
Define function $V:\mathbb{R}^+\times\mathbb{R}^l\rightarrow\mathbb{R}$ as
\[
V(t,x)=x^\top D(t)x.
\]
Then $V(t,x)$ is positive definite on $\mathbb{R}^l$ for all $t\geq 0$. Hence, by mean value theorem, 
\begin{align*}
\mathcal{L}&V(t,x-y)\\
=&\frac{\partial V}{\partial t}(t,x-y)+\left\langle f(t,x)-f(t,y),\frac{\partial V}{\partial x}(t,x-y)\right\rangle\\
&+\frac{1}{2}\trace\left[(g(t,x)-g(t,y))^\top\hess V(t,x-y)(g(t,x)-g(t,y))\right]\\
=&(x-y)^\top D'(t)(x-y)+2(f(t,x)-f(t,y))^\top D(t)(x-y)\\
&+\sum^{m}_{i=1}(g_i(t,x)-g_i(t,y))^\top D(t)(g_i(t,x)-g_i(t,y))\\
=&\Delta^\top D(t)\Delta+2\Delta^\top\left(\int^1_0\jac_f^\top(t,y+s\Delta)ds\right)D(t)\Delta\\
&+\Delta^\top\left[\sum^m_{i=1}\left(\int^1_0\jac_{g_i}^\top(t,y+s\Delta)ds\right)D(t)\left(\int^1_0\jac_{g_i}(t,y+s\Delta)ds\right)\right]\\
=&\Delta^\top\int^1_0\int^1_0\Bigg(D'(t)+2\jac_f^\top(t,y+s\Delta)D(t)\\
&~~~~~~~~~~~~~~~~~+\sum^m_{i=1}\jac_{g_i}^\top(t,y+s\Delta)D(t)\jac_{g_i}(t,y+\tau\Delta)\Bigg)dsd\tau\Delta\\
<&\Delta^\top\alpha(t)D(t)\Delta\\
=&\alpha(t)V(t,x-y),
\end{align*}
where $\Delta=x-y$. Take $Y(t)=V(t,X_\xi(t)-X_\eta(t))e^{-\int^t_0\alpha(s)ds}~~a.s.$ for $\xi,\eta\in L^2(\mathbb{P},\mathbb{R}^l)$. Thus, by It\^o's formula and the calculation in Theorem \ref{t2},
\[
\mathbb{E}Y(t)\leq\mathbb{E}V(0,\xi-\eta).
\] 
Therefore,
\[
\mathbb{E}V(t,X_\xi(t)-X_\eta(t))\leq\mathbb{E}V(0,\xi-\eta)e^{\int^t_0\alpha(s)ds}~~\forall~~t\in\mathbb{R}^+.
\]

Given any $t\in\mathbb{R}^+$, there must be a $k\in\mathbb{N}$ such that $t\in[kT,KT+T)$. Repeating the calculation in Step 2 of Theorem \ref{t1}, we have
\[
\mathbb{E}|X_\xi(t)-X_\eta(t)|^2\leq 3\mathbb{E}|X_\xi(kT)-X_\eta(kT)|^2e^{3(T+1)TL},
\]
where $L$ is the Lipschitz constant in {\bf(H2)}.

Note that for any $k\in\mathbb{N}$,
\[
\begin{split}
V(kT, X_\xi(kT)-X_\eta(kT))&=(X_\xi(kT)-X_\eta(kT))^\top D(kT)(X_\xi(kT)-X_\eta(kT))\\
&=(X_\xi(kT)-X_\eta(kT))^\top D(0)(X_\xi(kT)-X_\eta(kT))\\
&=V(0,X_\xi(kT)-X_\eta(kT))\\
&\leq \mathbb{E}V(0,\xi-\eta)e^{\int^{kT}_0\alpha(s)ds}.
\end{split}
\]
Moreover, since $D$ is real positive definite, there exist constants $\lambda,\Lambda>0$ such that for all $x\in\mathbb{R}^l$,
\[
\lambda|x|^2\leq V(0,x)\leq\Lambda|x|^2.
\]

Hence, for all $t\in[kT,kT+1)$,
\begin{align*}
\mathbb{E}|X_\xi(t)-X_\eta(t)|^2&\leq 3e^{3(T+1)TL}\mathbb{E}|X_\xi(kT)-X_\eta(kT)|^2\\
&\leq\frac{3}{\lambda}e^{3(T+1)TL+\int^{kT}_0\alpha(s)ds}\mathbb{E}|(\xi-\eta)^\top D(0)(\xi-\eta)|^2\\
&\leq\frac{3\Lambda}{\lambda}e^{3(T+1)TL+\int^{kT}_0\alpha(s)ds}\mathbb{E}|\xi-\eta|^2.
\end{align*}
Taking $a(t)=\frac{3\Lambda}{\lambda}e^{3(T+1)TL+\int^{t}_0\alpha(s)ds}$, then Condition ${\bf(H5)'}$ holds because of the locally integrability and boundedness of $\alpha$. Therefore, the proof is completed by Corollary \ref{cort2}.
\end{proof}

\subsection{Example}

Consider the following stochastic differential equation
\begin{equation}\label{eg}
\begin{split}
dX(t)&=\left[-a(t)|X(t)|^2X(t)-bX(t)+e(t)\right]dt+\left[c(t)X(t)+e(t)\right]dW(t)\\
     &=:f(t,X(t))dt+g(t,X(t))dW(t),
\end{split}
\end{equation}
where $a,c:\mathbb{R}^+\rightarrow\mathbb{R}$ are continuous $T$-periodic functions satisfying 
\begin{equation}\label{egcond}
\int^T_0\left(c^2(t)-2b\right)dt=-\varepsilon<0,
\end{equation}
$b\in(0,1)$ is constant and
\[
e(t)=\left(
\begin{matrix}
\sin\omega_1 t\\
\cos\omega_1 t\\
\vdots\\
\sin\omega_m t\\
\cos\omega_m t
\end{matrix}\right),~n=2m.
\]
Moreover, $a(t)\geq 0$ for all $t\geq 0$.

Thus we have
\[
Q=\left(
\begin{matrix}
\cos\omega_1T&\sin\omega_1T&0&0&\ldots&\ldots&\ldots\\
-\sin\omega_1T&\cos\omega_1T&0&0&&&\\
0&0&\cos\omega_2T&\sin\omega_2T&\ldots&\ldots&\ldots\\
0&0&-\sin\omega_2T&\cos\omega_2T&&&\\
\vdots&\vdots&&&\ddots&\vdots&\vdots\\
\vdots&\vdots&\ldots&\ldots&&\cos\omega_mT&\sin\omega_mT\\
\vdots&\vdots&&&&-\sin\omega_mT&\cos\omega_mT
\end{matrix}
\right)
\]

Let
\[
V(x)=\frac{1}{2}|x|^2.
\]
Hence, it can be verified that $f$ and $g$ are $(Q,T)$-affine periodic and {\bf(H6)} is satisfied. Although $f$ does not satisfy condition {\bf(H1)-(H2)}, we can still get global existence and boundedness of \eqref{eg} by Lyapunov's method. Let
\[
Z(t):=e^{bt}V(X(t)).
\]
Thus by It\^o's formula, we have
\begin{align*}
Z(t)=&V(X(0))+b\int^t_0e^{bs}V(X(s))ds\\
&+\int^t_0e^{bs}\left[f^\top(s,X(s))\cdot X(s)+\frac{1}{2}(g^\top\cdot g)(s,X(s))\right]ds\\
&+\int^t_0f^\top(s,X(s))\cdot X(s)dW(s).
\end{align*}
Therefore,
\begin{align*}
\mathbb{E}[Z(t)]=\mathbb{E}[V(X(0))]+\int^t_0e^{bs}\mathbb{E}\Bigg[&bV(X(s))+f^\top(s,X(s))\cdot X(s)\\
&+\frac{1}{2}(g^\top\cdot g)(s,X(s))\Bigg]ds.
\end{align*}
Note that for all $t\geq 0$ and $x\in\mathbb{R}^n$,
\begin{align*}
&bV(x)+f^\top(t,x)\cdot x+\frac{1}{2}(g^\top\cdot g)(t,x)\\
=&bV(x)+[-a(t)|x|^2x-bx+e(t)]\cdot x\\
&+\frac{1}{2}[c^2(t)|x|^2+2c(t)e^\top(t)\cdot x+|e(t)|^2]\\
\leq&-a(t)|x|^4+\left[-\frac{b}{2}+\frac{1}{2}+c^2(t)\right]|x|^2+\frac{3}{2}|e(t)|^2\\
\leq&(2C+1-b)V(x)+\frac{3m}{2},
\end{align*}
where $C=\sup\limits_{t\in[0,T]} c^2(t)$. Therefore, we have
\begin{align*}
&\mathbb{E}|X(t)|^2\\
=&2\mathbb{E}[V(X(t))]\\
\leq&\mathbb{E}|X(0)|^2+e^{-t}\int^t_0e^{bs}\mathbb{E}\left[(2C+1-b)V(X(t))+\frac{3m}{2}\right]ds\\
\leq&\mathbb{E}|X(0)|^2+\frac{3m}{2b}+\frac{2C+1-b}{2}e^{-t}\int^t_0e^{bs}\mathbb{E}|X(s)|^2ds.
\end{align*} 
Then by Gronwall's inequality \eqref{gron} with
\begin{align*}
&A(t)=\mathbb{E}|X(0)|^2+\frac{3m}{2b},\\
&B(t)=\frac{2C+1-b}{2}e^{-t},\\
&K(t)=e^{bt},
\end{align*}
we have
\begin{align}
&\mathbb{E}|X(t)|^2\nonumber\\
\leq&\mathbb{E}|X(0)|^2+\frac{3m}{2b}+\frac{2C+1-b}{2}\left(\mathbb{E}|X(0)|^2+\frac{3m}{2b}\right)e^{-t}\nonumber\\
&\cdot\int^t_0e^{bs}\frac{2C+1-b}{2(1-b)}\left[e^{(1-b)t}-e^{(1-b)s}\right]ds\nonumber\\
\leq&\mathbb{E}|X(0)|^2+\frac{3m}{2b}\nonumber\\
&+\frac{(2C+1-b)^2}{4b(b-1)}\left(\mathbb{E}|X(0)|^2+\frac{3m}{2b}\right)\cdot\left(1-b-e^{-bt}+be^{-t}\right)\nonumber\\
\leq&\left(\frac{(2C+1-b)^2}{4b(b-1)}+1\right)\left(\mathbb{E}|X(0)|^2+\frac{3m}{2b}\right).\label{egb}
\end{align}
This leads to existence and boundedness in mean square of solutions for \eqref{eg}.

Put $V$ in \eqref{lyaop}. Note that $V$ is autonomous with respect to $t$, we have
\begin{align*}
\mathcal{L}&V(x-y)\\
=&-(x-y)^\top\cdot[a(t)(|x|^2x-|y|^2y)+b(x-y)]\\
&+\frac{1}{2}c(t)(x-y)^\top\cdot c(t)(x-y)\\
=&-a(t)(x-y)^\top(|x|^2x-|y|^2y)+\left(\frac{c^2(t)}{2}-b\right)|x-y|^2.
\end{align*}

Let $H(x)=|x|^2x$. Then we have
\begin{align*}
\frac{\partial H}{\partial x}(x)=&|x|^2I_n+2\diag(x)\left(
\begin{matrix}
1&\cdots&1\\
\vdots&&\vdots\\
1&\cdots&1
\end{matrix}
\right)\diag(x)\\
=&|x|^2I_n+h(x),
\end{align*}
where $h$ is positive definite and thus $\frac{\partial H}{\partial x}$ is positive definite. Hence,
\begin{align*}
\mathcal{L}V(x-y)=&-a(t)(x-y)^\top\left(\int^1_0\frac{\partial H}{\partial(x)}(y+\theta(x-y))d\theta\right)(x-y)\\
&+\left(\frac{c^2(t)}{2}-b\right)|x-y|^2\\
\leq&\left(c^2(t)-2b\right)V(x-y).
\end{align*}

Let $\alpha(t)=c^2(t)-2b$. Since $c$ is continuous $T$-periodic, we have 
\[
\max_{t\in[0,T]}\int_{kT+t}^{kT+T+t}\alpha(s)ds
\]
is uniformly bounded for all $k\in\mathbb{N}$. Thus, by \eqref{egcond},
\[
\int^t_0\alpha(s)ds\rightarrow-\infty
\]
as $t\rightarrow\infty$. This means that condition {\bf (H7)} is fulfilled.

Hence, \eqref{eg} has a mean square asymptotically stable $(Q,T)$-affine periodic solution which is also unique in distribution by Theorem \ref{t2}.


\section*{Acknowledgements}
This work was supported by National Basic Research Program of China (grant No. 2013CB834100), National Natural Science Foundation of China (grant No. 11571065) and National Natural Science Foundation of China (grant No. 11171132).


\end{document}